\documentclass[3p,times]{elsarticle}

\usepackage{amssymb}
\usepackage{amsmath,amsthm}
\usepackage{amssymb,latexsym}
\usepackage{enumerate}

\numberwithin{equation}{section}

\newtheorem{theorem}{Theorem}[section]

\newtheorem{proposition}[theorem]{Proposition}
\newtheorem{lemma}[theorem]{Lemma}
\newtheorem{definition}[theorem]{Definition}

\newtheorem{The main theorem}[theorem]{The main theorem}

\theoremstyle{definition}

 \def\f{\mathcal F\text{-}} 
\allowdisplaybreaks

\begin{document}

\begin{frontmatter}



%% e.g. \dochead{17th International Conference on Dynamical Processes in Excited States of Solids}

\title {Weakly solutions to the complex Monge-Amp\`ere equation  on bounded  plurifinely hyperconvex domains}
%{The complex Monge-Amp\`ere equations on  plurifinely hyperconvex domains} 


\author[label1]{Nguyen Xuan Hong
}%\fnref{label1a}} \fntext[label1a]{This paper has been written during a stay of the first author at  the Vietnam Institute for Advanced Study in Mathematics. He wishes to thank the institution for their kind hospitality and support.   This research is funded by Ministry of education and training}
 \address[label1]{Department of Mathematics, Hanoi National University of Education, 136 Xuan Thuy Street, Cau Giay District, Hanoi, Vietnam} 
\ead{xuanhongdhsp@yahoo.com} \author[label3]{Hoang Van Can} \address[label3]{Department of Basis Sciences, University of Transport Technology, 54 Trieu Khuc, Thanh Xuan District, Hanoi, Vietnam}   \ead{ vancan.hoangk4@gmail.com}

\begin{abstract}
%In this paper,  we study the Dirichlet problem in  bounded $\mathcal F$-hyperconvex domains. 
%Let $\Omega \Subset \mathbb C^n$ be a bounded $\mathcal F$-hyperconvex domain and let  
Let  $\mu$ be a  non-negative measure  defined on bounded $\mathcal F$-hyperconvex domain $\Omega$.  We are interested in giving sufficient conditions on $\mu$ such that  we can find  a plurifinely plurisubharmonic function  satisfying  $NP (dd^c u)^n =\mu$ in $QB(\Omega)$. 
%can be solved.
%Some applications to weak convergence of sequence of m-subharmonic functions are also discussed.
\end{abstract}

\begin{keyword}
%% keywords here, in the form: keyword \sep keyword JMAA-15-3586 JMAA-14-3859
 plurifinely pluripotential theory    \sep plurifinely plurisubharmonic functions
 \sep complex Monge-Amp\`ere equations 


 \MSC[2010] 32U05 \sep 32U15
\end{keyword}

\end{frontmatter}

% \linenumbers


\section{Introduction}

Let $D$ be an open set in $\mathbb C^n$ and let   $ PSH^-(D)$ be the family of negative plurisubharmonic functions in $D$. 
The plurifine topology $\mathcal F$ on a  Euclidean open set $D$  is the smallest topology that makes all plurisubharmonic functions on $D$ continuous. 
Notions pertaining to the plurifine topology are indicated with the prefix $\mathcal F$ to distinguish them from notions pertaining to the Euclidean topology on $\mathbb C^n$.
For a set $A\subset \mathbb  C^n$ we write $\overline{A}$ for the closure of $A$ in the one point compactification
of $\mathbb C^n$, $\overline{A}^{\mathcal F}$ for the $\mathcal F$-closure of $A$ and $\partial _{\mathcal F}A$ for the $\mathcal F$-boundary of $A$. 

In 2003,
El Kadiri  \cite{K03} defined   the notion of $\mathcal F$-plurisubharmonic function in an $\mathcal F$-open subset of $\mathbb C^n$ and studied properties of such functions.   Later, 
El Marzguioui and Wiegerinck \cite{MW10}  proved the continuity properties of the plurifinely plurisubharmonic functions.  Next, 
El Kadiri,  Fuglede and Wiegerinck \cite{KFW11} proved  the most important properties of the  plurifinely plurisubharmonic functions.  
El Kadiri and Wiegerinck  \cite{KW14}  defined   the Monge-Amp\`{e}re operator on finite   plurifinely plurisubharmonic functions in $\mathcal F$-open sets.
They showed that it defines a non-negative measure  which vanishes on all pluripolar sets.  
Note that the measure is in general not a Radon measure, i.e. not Euclidean locally finite. %If $u$  is a  negative  $\mathcal F$-plurisubharmonic function   defined  on  $\mathcal F$-open set $\Omega$ then 
They also defined the non-polar part $NP(dd^c u)^n$ of $\mathcal F$-plurisubharmonic function  $u$  by
 $$
\int_A NP(dd^c u)^n = \lim_{j\to+\infty} \int_A  (dd^c \max(u,-j))^n, \ A\in QB(\Omega).
$$ 
El Kadiri and  Smit \cite{KS14} 
introduced   the notion of $\mathcal F$-maximal $\mathcal F$-plurisubharmonic functions  and studied properties of such functions.  
  Hong, Hai and Viet \cite{HHV17}   proved that $\mathcal F$-maximality is an $\mathcal F$-local notion for bounded $\mathcal F$-plurisubharmonic functions.  
 Trao, Viet and Hong  \cite{TVH}  
studied  the approximation of a negative $\mathcal F$-plurisubharmonic function  by an increasing sequence of plurisubharmonic functions.
They defined   the notion  of   bounded  $\mathcal F$-hyperconvex domain 
which extends the notion of bounded hyperconvex domain of $\mathbb C^n$  in a natural way. 
Recently,  Hong \cite{Hong17}  studied the Dirichlet problem in  $\mathcal F$-domain.  He 
proved  that under the suitable conditions, the  complex Monge-Amp\`ere equation can be solved.

The aim of this paper is to  study the  complex Monge-Amp\`ere equations  in  bounded $\mathcal F$-hyperconvex domains. 
Namely, we prove the following theorem.

\begin{theorem} \label{the1}
Let $\Omega  $ be a bounded $\mathcal F$-hyperconvex domain in $\mathbb C^n$ and let    $\mu$ be a  non-negative measure  on $QB(\Omega)$ which  vanishes  on all  pluripolar subsets of $\Omega$ such that  
\begin{equation}\label{ii}
\int_\Omega (-\psi) d \mu <+\infty \ , \text{  for some } \psi \in \mathcal F \text{-} PSH^-(\Omega).
\end{equation}
Then, there exists   $u\in  \mathcal F \text{-} PSH^- (\Omega) $ such that $NP(dd^c u)^n = \mu$ on $QB(\Omega)$. 
\end{theorem}

The paper is organized as follows. In Section 2, we recall some notions of  plurifine   pluripotential theory.
Section 3 is devoted to prove  Theorem \ref{the1}. 
In Section 4,  we prove a result to  show that the condition \eqref{ii} in Theorem \ref{the1} is sharp.

\section{Preliminaries}
Some elements of  plurifine potential  theory  that will be used  throughout  the paper can be  found  in \cite{ACCH}-\cite{W12}.  First, we  recall  the following definitions (see \cite{K03},  \cite{KFW11},  \cite{MW10}, \cite{W12}).

\begin{definition} {\rm  
Let $\Omega$ be  an $\mathcal F$-open   subset of $\mathbb C^n$. A function $u : \Omega \to  [-\infty, +\infty) $  is said to be  $\mathcal F$-plurisubharmonic if $u$ is $\mathcal F$-upper semicontinuous and for every complex line $l$ in $\mathbb C^n$, the restriction of $u$ to any $\mathcal F$-component of the finely open subset $l \cap \Omega$  of $l$ is either finely subharmonic or $\equiv -\infty$.

The set of all negative  $\mathcal F\text{-}$plurisubharmonic functions defined in $\mathcal F\text{-}$open set $\Omega$ is denoted by $\mathcal F\text{-}  PSH^-(\Omega)$.
}\end{definition}

\begin{definition}{\rm 
Let $\Omega\subset \mathbb C^n$ be an $\mathcal F$-open set   and let $u  \in \mathcal F\text{-} PSH^-(\Omega)$. Denote by    $QB(\mathbb C^n)$  the measurable space on $\mathbb C^n$ generated by the Borel sets and the pluripolar subsets of  $\mathbb C^n$ and $QB(\Omega)$  is  the trace of $QB(\mathbb C^n)$ on $\Omega$.

(i) If $u $ is finite then there exist a pluripolar set $E\subset \Omega$, a sequence of $\mathcal F$-open subsets  $\{O_j\}$   and   
plurisubharmonic functions  $f_{j}, g_{j}$ defined in Euclidean neighborhoods of $\overline O_j$  such that   
 $\Omega=E \cup \bigcup_{j =1}^\infty O_j $ and   $u=f_{j} -g_{j}$ on $O_j $.
The Monge-Amp\`ere measure   $(dd^c u)^n$ on $QB(\Omega)$  is defined by  
$$
\int_A  (dd^c u)^n     := \sum_{j=1}^\infty \int_{A\cap ( O_j \backslash \bigcup_{k=1}^{j-1} O_k ) }  (dd^c ( f_{j} -g_{j}))^n, \ \ A\in QB(\Omega).
$$ 

(ii) The non-polar part $NP(dd^c u)^n$  is defined by
 $$
\int_A NP(dd^c u)^n = \lim_{j\to+\infty} \int_A  (dd^c \max(u,-j))^n, \ A\in QB(\Omega).
$$ 
}\end{definition}

\begin{proposition}
Let $\Omega$ be an $\mathcal F$-domain in $\mathbb C^n$ and let  $v,w\in \mathcal F\text{-}PSH(\Omega)$ be finite with  $w\leq -v$. 
Then, for every   measure $\mu$ on the  $QB(\Omega)$ with $0\leq \mu \leq (dd^c w)^n$, there exists a finite  plurifinely plurisubharmonic function $u$ defined on $\Omega$ such that   $w\leq u\leq -v$  and 
$$(dd^c u)^n =\mu \text{ in } QB(\Omega).$$
\end{proposition}
\begin{proof}
See \cite{Hong17}.
\end{proof}

\begin{definition}{\rm  Let $\Omega$ be an $\mathcal F \text{-}  $open set in $\mathbb C^n$ and let    $u\in \mathcal F \text{-}   PSH(\Omega)$.  We say that $u$  is $\mathcal F \text{-}  $maximal  in $\Omega$  if for every bounded $\mathcal F \text{-}  $open set $G$ of $\mathbb C^n$ with $\overline G\subset \Omega$, and for every function  $v\in \mathcal F \text{-}   PSH(G)$ that is bounded from above on $G$ and extends $\f$upper semicontinuously to $\overline{G}^{\mathcal F}$ with   $v\leq u$ on  $\partial_{\mathcal F} G$ implies   $v\leq u$ on $G$.  
}\end{definition}

\begin{proposition}
 If $u$ is  a finite $\mathcal F$-maximal $\mathcal F$-plurisubharmonic function defined  on an $\mathcal F$-domain then $(dd^c u)^n = 0$.
\end{proposition}
\begin{proof}
See Theorem 4.8 in \cite{KS14}.
\end{proof}

\begin{proposition}
Let $\Omega$ be an $\mathcal F$-open set in $\mathbb C^n$ and assume that $u\in \mathcal F \text{-} PSH(\Omega)$ is bounded. 
Then,  $u$ is  $\mathcal F$-maximal in $\Omega$
if and only if  $(dd^c u)^n =0$ on $QB(\Omega)$.
\end{proposition}
\begin{proof}
See  \cite{HHV17}.
\end{proof}

We now recall   the definition of bounded $\mathcal F$-hyperconvex domain $\Omega$ and the class $\mathcal F_p(\Omega)$  which is similar to the class introduced in \cite{Ce98} for the case  of a bounded hyperconvex domain (see \cite{TVH}).

\begin{definition}  
{\rm 
(i) A   bounded $\mathcal F$-domain $\Omega$ in $\mathbb C^n$    is  called  $\mathcal F$-hyperconvex if there exist a negative bounded plurisubharmonic function $\gamma_\Omega$ defined in a bounded hyperconvex domain $\Omega'$ such that $\Omega =\Omega' \cap \{\gamma_\Omega >-1\}$ and $-\gamma_\Omega $ is $\mathcal F$-plurisubharmonic  in $\Omega$.

(ii) We say that  a bounded negative  $\mathcal F$-plurisubharmonic function $u$ defined on a bounded $\mathcal F$-hyperconvex domain $\Omega$ belongs to $\mathcal E_0(\Omega)$ if  $\int_\Omega (dd^c u)^n<+\infty$ and   for every $\varepsilon>0$  there exists $\delta>0$ such that  
$$
\overline{\Omega\cap  \{u<- \varepsilon\} } \subset  \Omega' \cap\{\gamma_\Omega >-1+\delta\}.
$$ 

(iii) Denote by  $\mathcal F_p (\Omega)$, $p>0$   the family of negative $\mathcal F$-plurisubharmonic functions $u$ defined on $\Omega$ such that there exist a decreasing sequence $\{u_j\} \subset \mathcal E_0(\Omega)$ that converges pointwise to $u$ on $\Omega$ and 
$$\sup_{j\geq 1} \int_\Omega (1+(-u_j)^p)  (dd^c u_j)^n<+\infty.$$
} \end{definition}

%Example 3.3 in \cite{TVH} showed that there exists a  bounded  $\mathcal F$-hyperconvex domain that has no Euclidean interior point exists.
%We need the following. 

\begin{proposition} \label{lem0}
Let $\Omega \Subset \mathbb C^n$ be a bounded $\mathcal F$-hyperconvex domain. Assume that  $u\in \mathcal F_1  (\Omega)$ is bounded  and $v\in \mathcal F\text{-}PSH^-(\Omega)$ such that   $ (dd^c u)^n \leq  (dd^c v)^n$ in $\Omega \cap   \{v>-\infty\}$.  
Then, $u \geq v$ in $\Omega$. 
\end{proposition}

\begin{proof}
Without loss of generality  we can assume that $-1\leq u \leq 0$ on $\Omega$. Let $j \in \mathbb N^*$ and define 
$$
v_j:= (1+\frac{1}{j}) (v-\frac{1}{j}) \text{ in } \Omega.
$$
Choose $p>0$ such that 
$ j^p <  1+\frac{1}{j} $.
It is easy to see that 
\begin{align*}
(1+ (-u)^p) (dd^c u)^n 
\leq 2 (dd^c u)^n  
\leq 2 (dd^c v)^n 
\leq (1+ (-v_j)^p) (dd^c v_j)^n 
\text{ on } \Omega \cap   \{v_j>-\infty\}.
\end{align*}
Since $u$ is bounded, so $u\in \mathcal F_p (\Omega)$, and hence,  Proposition 4.4 in \cite{TVH} implies that 
$u \geq v_j$ in $\Omega$. 
Letting $j\to+\infty$ we conclude  that $u \geq v$ in $\Omega$. The proof is complete. 
\end{proof}

\section{The complex Monge-Amp\`ere equations}

\begin{lemma} \label{lem1}
Let $D $ be a bounded hyperconvex domain in $\mathbb C^n$  and let $ \Omega \Subset D$ be  a bounded $\mathcal F$-hyperconvex domain. 
Assume that $u \in \mathcal E_0(\Omega)  $ and define 
$$
w:=\sup\{ \varphi \in PSH^-(D): \varphi \leq u \text{ on } \Omega\}.
$$
Then, $w \in \mathcal E_0(D)$ and $(dd^c w)^n \leq 1_\Omega (dd^c u)^n$ in $D$. 
\end{lemma}

\begin{proof}
It is easy to see that $w \in \mathcal E_0(D)$. Without loss of generality we can assume that $- \frac{1}{2}\leq u < 0$ in $\Omega$, and hence, $- \frac{1}{2}\leq w < 0$ in $D$. 
First, we claim that 
\begin{equation} \label{eq999-21}
 (dd^c w)^n \leq  (dd^c u)^n \text{ on } \Omega \cap  \{w=u\}.
\end{equation}
Indeed, 
let $\Omega'$ be a bounded hyperconvex domain in $\mathbb C^n$ and let  $\gamma_\Omega \in PSH^-(\Omega') \cap L^\infty( \Omega')$ such that   $\Omega= \Omega' \cap \{\gamma_\Omega >-1\} $ and $-\gamma_\Omega  \in \mathcal F\text{-}PSH(\Omega)$. 
Choose $\varepsilon, \delta>0$ such that $\sup_\Omega w <-2\varepsilon$   and  
$$\overline{\Omega\cap  \{u<- \varepsilon\} } \subset  \Omega' \cap\{\gamma_\Omega >-1+  2 \delta\}.$$ 
Let $j$ be an integer number with $j \varepsilon >1$.
Proposition 2.3 in \cite{KS14} states that the functions 
$$f := 
\begin{cases}
\max(- \frac{1}{\delta}, u+  \frac{1}{\delta}  \gamma_\Omega ) & \text{ in } \Omega
\\ -  \frac{1}{\delta}& \text{ in } \Omega' \backslash \Omega
\end{cases}  
\text{ and }
f _j := 
\begin{cases}
\max(- \frac{1}{\delta}, \max(u, w+\frac{1}{2j}) +  \frac{1}{\delta}  \gamma_\Omega ) & \text{ in } \Omega
\\ -  \frac{1}{\delta}& \text{ in } \Omega' \backslash \Omega
\end{cases}  
$$  
are   $\mathcal F$-plurisubharmonic   and Proposition 2.14 in \cite{KFW11} states that $f , f_j \in  PSH(\Omega')$ because $\Omega' $ is a Euclidean open set.
Since   $u=f- \frac{1}{\delta }  \gamma_\Omega $, $\max(u, w +\frac{1}{2j}) =f_j - \frac{1}{\delta }  \gamma_\Omega $  in $\{ \gamma_\Omega  >-1+\delta \}$, by \cite{BT87} we have 
\begin{equation} \label{eq0-01}
\lim_{j \to +\infty} \int_\Omega \chi (dd^c \max(u, w+\frac{1}{j}))^n  = \int_\Omega \chi (dd^c u)^n
\end{equation} 
for every bounded $\mathcal F $-continuous function $\chi$ with compact support on $\{ \gamma_\Omega  >-1+\delta \}$.
Let $K\subset \Omega \cap \{w=u\}$ be a compact set. Since $\Omega \cap \{w=u\} \subset \Omega \cap  \{u<-\varepsilon\}  \subset \{ \gamma_\Omega  >-1+  2 \delta \}$, there exists    a decreasing sequence of bounded $\mathcal F $-continuous functions  $\{\chi_k\}$   with compact support on $\{ \gamma_\Omega  >-1+\delta \}$ such that $\chi_k \searrow 1_K$ as $k \nearrow +\infty$. By
Theorem 4.8 in \cite{KW14}  we conclude by \eqref{eq0-01} that 
\begin{align*}
\int_K    (dd^c w)^n 
\leq \lim_{j \to +\infty} \int_\Omega \chi _k(dd^c \max(u, w+\frac{1}{j}))^n  
= \int_\Omega \chi_k (dd^c u)^n, \ \forall k \geq 1.
\end{align*}
Letting $k\to+\infty$, we obtain that 
\begin{align*}
\int _K   (dd^c w)^n 
\leq   \int_K (dd^c u)^n .
\end{align*}
Therefore, 
$
 (dd^c w)^n \leq  (dd^c u)^n \text{ on } \Omega \cap  \{w=u\}.
$
This proves the claim.   
Now, since $u$ is $\mathcal F$-continuous on $\Omega$, it follows that  the function 
$$
h := 
\begin{cases}
u & \text{ on } \Omega
\\ 0 & \text{ in } D \backslash \Omega 
\end{cases}
$$
is $\mathcal F$-continuous on $D$, and hence,  
$$U:= D\cap \{w<h \} \text{  is }\mathcal F\text{-open set}.$$ 
Let $z\in U$ and let $a\in \mathbb R$ be such that $w(z) < a< h (z)$. Let $V$   be a connected component of the $\mathcal F$-open set $D\cap \{w<a\} \cap \{h >a\}$ which contains the point $z$. 
We claim that  $w $ is $\mathcal F$-maximal in $V$. 
Indeed, let 
$G$ be a bounded  $\mathcal F$-open set in  $\mathbb C^n$ with $\overline G\subset V$ and let   $v\in \mathcal F \text{-} PSH(G)$ such that $v$ is  bounded from above  on $G$, extends $\mathcal F$-upper semicontinuously to $\overline{G}^{\mathcal F}$ and    $v\leq w$ on  $\partial_{\mathcal F} G$. Since $D$ is a Euclidean open set,  Proposition 2.3 in \cite{KS14} and Proposition 2.14 in \cite{KFW11} imply that  the function
$$\varphi  := \begin{cases} 
\max(w , v  )  
& \text{ on } G 
\\ w & \text{ on } D \backslash G 
\end{cases}$$ 
is  plurisubharmonic in $D$.  
Because  $\overline{G} \subset V \subset D\cap \{w<a\}$, we infer that $\varphi <a $ on $\overline{G} $, and therefore, $\varphi \leq h$ in $D$. It follows that 
$\varphi =w$ in $D$. Thus, $v\leq w$ in $G$, and hence, $w$ is $\mathcal F$-maximal in $V$. This proves the claim. 
Therefore,   $w$ is $\mathcal F$-locally $\mathcal F$-maximal in $U$.  Theorem 1 in \cite{HHV17} implies that 
$$(dd^c w)^n =0 \text{  on } U.$$
Combining this with \eqref{eq999-21} we arrive at 
$(dd^c w)^n \leq 1_\Omega (dd^c u)^n$ in $D$. 
The proof is complete.
\end{proof}

We now able to give the proof of theorem \ref{the1}.

\begin{proof}[Proof of Theorem \ref{the1}]
Without loss of generality we can assume that  $\psi \in \mathcal E_0(\Omega)$ and $-1\leq \psi < 0$ in $\Omega$. 
Let $r>0$ be such that $\Omega \Subset \mathbb B(0,r)$. Let $j\geq 1$ be an integer number.
Since 
$$
\int_{\mathbb B(0,r)} 1_{\Omega \cap \{ - \frac{1}{j} \leq \psi< - \frac{1}{j+1}\}}  d\mu 
\leq (j+1) \int_\Omega (-\psi) d\mu <+\infty,
$$
Theorem 6.2 in \cite{Ce2} implies that there exist $\psi _j \in \mathcal E_0(\mathbb B(0,r))$ and $0\leq f_j \in L^1 ( (dd^c \psi_j)^n) $ such that 
$$
1_{\Omega \cap \{ - \frac{1}{j} \leq \psi< - \frac{1}{j+1}\}}  \mu 
= f_j (dd^c \psi_j)^n \text{ in } \mathbb B(0,r).
$$
Let $a_j>0$ be such that 
$$
\sum_{k=1}^j \psi_k > -a_j \text{ in } \mathbb B(0,r).
$$
Thanks to Theorem 4.8 in \cite{KW14}  we have
\begin{align*}
(dd^c \max(\sum_{k=1}^j \psi_k ,  a_j (j+1) \psi  ))^n 
& \geq 1_{\Omega \cap \{ \psi < -\frac{1}{j+1}\}} (dd^c (\sum_{k=1}^j \psi_k  ))^n 
\\& \geq   \sum_{k=1}^j 1_{\Omega \cap \{ -\frac{1}{k} \leq \psi < -\frac{1}{k+1}\}} (dd^c \psi_k )^n \geq \sum_{k=1}^j \min (f_k, j) (dd^c \psi_k)^n \
\text{ on } QB(\Omega).
\end{align*} 
By Theorem 1.1 in \cite{Hong17} we can find $u_j \in \mathcal F \text{-} PSH (\Omega)$ such that  $\max(\sum_{k=1}^j \psi_k ,  a_j (j+1) \psi  )  \leq u \leq 0$ on $\Omega$ and
$$
(dd^c u_j)^n =\sum_{k=1}^j \min (f_k, j) (dd^c \psi_k)^n \
\text{ on } QB(\Omega).
$$
Since $\max( \sum_{k=1}^j \psi_k ,  a_j (j+1) \psi  ) \in \mathcal E_0(\Omega)$,  
Proposition 3.4 in \cite{TVH} implies that $u_j \in \mathcal E_0(\Omega)$, and hence, 
Proposition \ref{lem0} states that $u_j \geq u_{j+1}$ in $\Omega$ because $(dd^c u_j)^n \leq (dd^c u_{j+1})^n $ on $QB(\Omega)$. 
Put
$$u:=\lim_{j\to+\infty} u_j \text{ on } \Omega.$$  
We claim that $u \not \equiv -\infty$ in $\Omega$. 
Indeed, without loss of generality we can assume that $G:=\{\psi <-\frac{1}{2}\} \neq \emptyset$. We set 
$$
v_j:= \sup\{ \varphi\in \mathcal F \text{-} PSH ^- (\Omega)  : \varphi \leq u_j \text{ on }G \}.
$$
Then,    $v_j \in \mathcal E_0(\Omega)$, $u_j \leq v_j <0$ in $\Omega$ and $v_j=u_j$ on $G$.
By Proposition 3.1 in \cite{KS14} we have 
$u_j$ is $\mathcal F$-maximal on $\Omega \cap \{\psi>-\frac{1}{2}\}$, and hence, $(dd^c v_j)^n=0$ on $\Omega \cap \{\psi>-\frac{1}{2}\}$.
Therefore, using Proposition 3.4 in \cite{TVH} we infer that 
\begin{equation} \label{eq1}
\begin{split}
\int_\Omega (dd^c v_j)^n 
& =\int_{\Omega \cap \{ \psi \leq - \frac{1}{2}\} } (dd^c v_j)^n 
\leq 2 \int_{\Omega } (- \psi )(dd^c v_j)^n 
\\& \leq 2 \int_{\Omega } (- \psi )(dd^c u_j)^n 
 = 2 \sum_{k=1}^j  \int_{\Omega } (- \psi ) \min (f_k, j) (dd^c \psi_k)^n
\\ & \leq  2 \sum_{k=1}^j  \int_{\Omega } (- \psi )  1_{\Omega \cap \{ -\frac{1}{k} \leq \psi < - \frac{1}{k+1} \}} d \mu  
\leq 2 \int_\Omega (-\psi) d\mu.
\end{split}
\end{equation}
Lemma \ref{lem1} states that the function 
$$
w_j:=\sup\{ \varphi \in PSH^-(\mathbb B(0,r)): \varphi \leq v_j \text{ on } \Omega\}
$$
belongs to  $  \mathcal E_0(\mathbb B(0,r))$ and 
$$(dd^c w_j)^n \leq 1_\Omega (dd^c v_j)^n \text{ in } \mathbb B(0,r).$$  
Hence, by the hypotheses and using \eqref{eq1}  we get 
\begin{align*}
\sup_j \int_{\mathbb B(0,r)} (dd^c w_j)^n 
\leq \sup_j \int_{\Omega} (dd^c v_j)^n 
\leq   2 \int_\Omega (-\psi) d\mu <+\infty.
\end{align*}
This implies that 
$w:=\lim_{j\to+\infty} w_j$ is plurisubharmonic function in $ \mathbb B(0,r) $.   
Thanks to Theorem 2.3 in  \cite{KFW11} we arrive at 
$\mathbb B(0,r) \cap \{ w=-\infty\}$ has no $\mathcal F$-interior point. Moreover, since $w\leq v =u $ on $G$
so $u\not \equiv -\infty$ on $G$. Thus, $u\not \equiv -\infty$ on $\Omega$. 
This proves the claim, and therefore, $u \in \mathcal F \text{-} PSH(\Omega)$.
Since $u_j \searrow u$ on $\Omega \cap \{u>-\infty\}$ as $j\nearrow +\infty$, by Theorem 4.5 in \cite{KS14}  we have  the sequence of measures $(dd^c u_j)^n$ converges $\mathcal F$-locally vaguely to $(dd^c u)^n$ on $\Omega \cap \{u>-\infty\}$.
Moreover, since 
$$
(dd^c u_j)^n =\sum_{k=1}^j \min (f_k, j) (dd^c \psi_k)^n \nearrow   \sum_{k=1}^\infty  f_k  (dd^c \psi_k)^n  =\mu \text{ on } \Omega  ,
$$
it follows that 
$$
(dd^c u)^n = \mu \text{ on } QB(\Omega \cap \{u>-\infty\} ) .
$$
Thus, 
$
NP(dd^c u)^n = \mu \text{ on } QB(\Omega) .
$
The proof is complete.
\end{proof}

\section{Measures without   solution  exists}

%First we show that there exists  a   non-negative measure $\mu$ defined on  $\mathcal F$-hyperconvex domain $\Omega$ which satisfy the hypotheses of Theorem \ref{the1}, however, $1_\Omega \mu$ is not a Radon measure in $\mathbb C^n$, and hence, we can not be solved  the  complex Monge-Amp\`ere equation  $(dd^c u)^n = 1_\Omega \mu$ in the sense of \cite{Bl} on any Euclidean open neighborhoods of $\Omega$. 

First we prove the following. 

\begin{proposition}
Let $\Omega $ be a   bounded $\mathcal F$-hyperconvex domain in $\mathbb C$ that has no Euclidean interior point exists.
Then, there exists a non-negative measure $\mu$ on $QB(\Omega)$ such that

(i) $\mu$ vanishes on all pluripolar subsets of $\Omega$;

(ii) $\int_\Omega (-\psi) d \mu <+\infty $, for some negative finely subharmonic function $\psi$ in $\Omega$;  

(iii)  There is no subharmonic function $w$ defined on Euclidean open neighborhood of $\Omega$ satisfying  $$
NP(dd^c w) =\mu \text{ on } QB(\Omega).
$$ 
\end{proposition}

\begin{proof}
Let  $\gamma_\Omega$ be  a negative bounded subharmonic function defined in a bounded hyperconvex domain $\Omega'$ such that  $\Omega =\Omega' \cap \{\gamma_\Omega >-1\}$ and  $-\gamma_\Omega $ is   finely subharmonic  in $\Omega$.
Since $\Omega$ has no Euclidean interior point exists, we can find $\{a_j \}  \subset \Omega$ such that $a_j\to a \in \partial _{\mathcal F} \Omega \cap \Omega'$ and  
$$ -1< \gamma_\Omega (a_{j+1}) < \gamma_\Omega (a_j) < -1+ \frac{1}{2^{j} }.$$  
Theorem 4.14 in \cite{ACCH} implies that there exist $u_j \in \mathcal F(\Omega')$ be such that 
$$
dd^c u_j   = \delta _{a_j} \ \text{ in  } \Omega',
$$
where  $\delta_{a_j}$ denotes  the Dirac measure at $a_j$.  
Let $r_j\in (0, \frac{1}{2^j})$ and let $\varphi_j \in SH^-(\mathbb B(a_j, 2 r_j))$ be such that $\varphi_j (a_j) >-1$ and 
$$ \mathbb B(a_j, 2r_j ) \cap \{\varphi_j >-2\} \subset  \Omega \cap \{\gamma_\Omega (a_{j+1}) < \gamma_\Omega < -1+ \frac{1}{2^{j} } \}.$$ 
%We set   $v_j:=\lim_{k\to+\infty} u_{j,k}$, where 
We set 
$$
u_{j,k}:= \sup \{ \varphi \in SH^-(\Omega') : \varphi \leq \max( u_j ,-k)  \text{ on  }  \mathbb B(a_j,r_j) \cap \{\varphi_j >-1\}  \} .
$$
Since $\mathbb B(a_j,r_j) \cap \{\varphi_j>-1\}$ is $\mathcal F$-open set, by Corollary 3.10 in \cite{KFW11} we infer that 
$u_{j,k} \in  SH^-(\Omega')$.
Proposition 3.2 in \cite{KS14} implies that   $u_{j, k_j}$ is $\mathcal F$-maximal on the $\mathcal F$-interior   of $\Omega' \backslash ( \mathbb B(a_j,r_j) \cap \{\varphi_j >-1\} )$, and hence,   
by Theorem 4.8 in \cite{KS14} (also see \cite{HHV17}) we get 
\begin{equation} \label{8.42.12.12}
 dd^c u_{j, k_j}   =0  \ \text{ on } \Omega' \backslash ( \overline{ \mathbb B} (a_j,r_j) \cap \{\varphi_j \geq -1\} ).
%  \Omega' \cap  ( \{ \gamma_\Omega > -1+ \frac{1}{2^{j} } \}  \cup \{ \gamma_\Omega < \gamma_\Omega (a_{j+1})\}) .
\end{equation}

We now claim that $u_{j,k} \searrow u_j$ in $\Omega'$ as $k\nearrow +\infty$. 
Indeed, since $u_j \leq u_{j,k+1} \leq u_{j,k} $ in $\Omega'$ so  $v_j:= \lim_k u_{j,k} \in SH^-(\Omega')$ and   $v_j \geq u_j$. Because  $u_j=v_j$ on $ \mathbb B(a_j,r_j) \cap \{\varphi_j >-1\}$,  by  Lemma 4.1 in \cite{ACCH} and  
Theorem 1.1 in \cite{HHiep16}  we get
\begin{align*}
 dd^c v_{j} 
\leq    dd^c u_{j} 
=1 _{\{a_j\}} dd^c u_{j} 
=1 _{\{a_j\}} dd^c v_{j} 
\leq  dd^c v_{j}  \ \text{ in } \Omega'.
\end{align*}
This implies that  $ dd^c v_{j}  =  dd^c u_{j}  = \delta_{a_j}$ in $\Omega'$. 
According to Theorem 3.6 in \cite{ACCH} we have 
$v_j=u_j$ in $\Omega'$, and hence, $u_{j,k} \searrow u_j$ on $\Omega'$ as $k\nearrow +\infty$. This proves the claim.  Therefore, Corollary 3.4 in \cite{ACCH} implies that 
\begin{align*}
\lim_{k\to+\infty} \int _{\Omega'} dd^c u_{j,k}
=  \int _{\Omega'} dd^c u_{j} = 1.
\end{align*}
Let  $k_j\geq 1$ be such that  
\begin{equation} \label{8.43.12.12}
 \frac{1}{2} \leq  \int _{\Omega'} dd^c u_{j, k_j}  \leq  1 .
\end{equation}
We set $\psi := -1-\gamma_\Omega  $ and 
$$\mu:= \sum_{j \geq 1} dd^c u_{j, k_j}   \text{ on } QB(\Omega ).
$$ 
Then, $\psi$ is negative finely subharmonic function on $\Omega$ and $\mu$ vanishes on all pluripolar subsets of $\Omega$.
Since 
$$ \overline{ \mathbb B} (a_j,r_j) \cap \{\varphi_j \geq -1\} 
\subset \mathbb B  (a_j, 2 r_j) \cap \{\varphi_j > -2\}
\subset 
 \Omega \cap    \{ \gamma_\Omega < -1+ \frac{1}{2^{j} } \}   ,$$
 by   \eqref{8.42.12.12} and \eqref{8.43.12.12} we arrive at 
\begin{align*}
\int_\Omega (-\psi ) d \mu 
& =\sum_{j\geq 1} \int_\Omega (1+ \gamma_\Omega )  dd^c u_{j,k_j} 
\\ & =\sum_{j\geq 1} \int_{ \Omega  \cap   \{    \gamma_\Omega <  -1+ \frac{1}{2^{j} } \}}  (1+ \gamma_\Omega )  dd^c u_{j,k_j} 
\\& \leq \sum_{j\geq 1}  \frac{1}{2^{j} } \int_{ \Omega' }  dd^c u_{j,k_j}  
\leq 1.
\end{align*}
We now assume that there exists a  subharmonic function $w$ defined on Euclidean open neighborhood $O$ of $\Omega$ such that 
$$NP(dd^c w)  =\mu \ \text{ on } QB(\Omega).$$
Let $r>0$ and let $j_0 \in \mathbb N$ be such that $\mathbb B(a_j, 2 r_j) \Subset \mathbb B(a,r) \Subset O$ for all $j\geq j_0$.  
Again by \eqref{8.42.12.12} and \eqref{8.43.12.12} we get 
\begin{align*}
+\infty
& > \int_{\mathbb B(a,r)}  dd^c w 
 \geq \int_{\mathbb B(a,r) \cap \Omega}  d \mu  
\\ & \geq \sum_{j\geq j_0} \int_{\overline {\mathbb B}(a_j,r_j) \cap \{\varphi_j \geq -1\} }  dd^c u_{j,k_j}
\\ & = \sum_{j\geq j_0} \int_{ \Omega '}  dd^c u_{j,k_j}
 \geq \sum_{j\geq j_0}    \frac{1}{2}
=+\infty.
\end{align*}
This is impossible. The proof is complete. 
\end{proof}

We now recall that a  negative plurisubharmonic functions $u$ defined on bounded hyperconvex domain $\Omega'$ belongs to $\mathcal F(\Omega')$ if  there exist a decreasing sequence $\{\varphi _j\} \subset \mathcal E_0(\Omega')$ that converges pointwise to $u$ on $\Omega'$ and 
$$\sup_{j\geq 1} \int_{ \Omega '} (dd^c \varphi_j )^n<+\infty.$$ 
The following  result  shows that  the condition \eqref{ii} in Theorem \ref{the1} is sharp.

\begin{proposition} 
Let $n$ be an integer number with  $n\geq 2$. Then, there exist a bounded $\mathcal F$-hyperconvex domain $\Omega \subset \mathbb C^n$ and a non-negative measure $\mu$ on $QB( \Omega)$  such that 

(i) $\Omega$ has no Euclidean interior point exists;

(ii) $\mu$ vanishes on all pluripolar subsets of $ \Omega$;

(iii) $\int_{  \Omega}(-\psi )^p d\mu<+\infty \text{ for all } p>1$, for some $ \psi \in \mathcal F \text{-} PSH^-(\Omega)$;

(iv) There is no function $w\in\mathcal  F \text{-} PSH^-(  \Omega )$ satisfying $NP(dd^c w)^n=\mu$ on $QB( \Omega)$. 
\end{proposition}

\begin{proof}
Let $\Delta$ be a unit disc in $\mathbb C$. 
By Example 3.3 in \cite{TVH} we can find  a bounded  $\mathcal F$-hyperconvex domain  $D \subset \Delta$  and an increasing  sequence of negative   subharmonic functions  $ \rho_j $  defined on  bounded  hyperconvex domains $D_j$ such that  $D\subset D_{j+1} \subset D_j$, $D$ has no Euclidean interior point exists and   $\rho_j\nearrow \rho \in \mathcal E_0(D)$ a.e. on $D$. Let $a\in D$.
Thanks to Theorem 4.14 in \cite{ACCH} we can find an increasing sequence  of negative subharmonic functions $u_j \in \mathcal F(D_j)$ such that $u_j \leq u_{j+1}$ on $D_{j+1}$ and  
$$dd^c u_j  =\delta_a \text{ on } D_j, $$
where  $\delta_a$ denotes  the Dirac measure at $a$.
Let  $u$ be  the least $\mathcal F$-upper semicontinuous majorant  of $(\sup_{j\geq 1} u_j)$ in $D$. Then, $u_j\nearrow u$ a.e. on $D$ as $j\nearrow +\infty$. 
Let  $v \in \mathcal F(\Delta^{n-1})$ be such that $\lim_{z\ni \Delta^{n-1} \to \partial \Delta^{n-1}} v(z)=0$ and 
$$
(dd^c v)^{n-1} = \delta_{o} \text{ on } \Delta^{n-1}, \text{ where } o \text{ is the origin of } \mathbb C^{n-1}. 
$$
We set $\Omega:= D\times \Delta^{n-1}$ and
$\mu:= \sum_{k\geq 1} (dd^c w_k )^{n}$  on $\Omega $, where 
$$
w_k(t,z) :=  \max(2^k u(t), kv(z),-1)   , \ \ (t,z) \in D \times \Delta^{n-1}.
$$
It is easy to see that $\Omega$ is bounded $\mathcal F$-hyperconvex domain that has no Euclidean interior point exists.
Using Theorem 4.6 in \cite{KFW11} we obtain that 
$w_k \in \mathcal F \text{-} PSH^-(\Omega) \cap L^\infty (\Omega)$, and therefore,  $\mu$ vanishes on all pluripolar subsets of $ \Omega$.  
We now claim that 
$$
\int\limits_{D \times \Delta ^{n-1}} (- \max(u,v))^p d \mu <+\infty, \ \forall p>1.
$$
Indeed, since $dd^c  \max (  u_j  , - \frac{1}{2^{k}})=0$ on $D_j \cap  \{u_j  \neq - \frac{1}{2^{k}}\} $, by  Corollary 2.1 in \cite{ACH} and Corollary 4.2 in \cite{ACH} we get  
\begin{align*}
&\int\limits_{D_j  \times \Delta ^{n-1}} (- \max(u_j,v) )^p (dd^c \max(2^k u_j , kv,-1) )^{n} 
\\& = \int\limits_{D_j  \times \Delta ^{n-1}}(- \max(u_j,v) )^p dd^c  \max ( 2^{k} u_j  , -1)  \wedge (dd^c  \max ( k v  , -1) )^{n-1} 
\\& =  2^k k^{n-1}  \int\limits_{D_j  \times \Delta ^{n-1}} (- \max(u_j,v) )^p dd^c  \max (  u_j  , - \frac{1}{2^{k}})  \wedge (dd^c  \max (  v  , - \frac{1}{k}) )^{n-1}
\\& \leq  2^{-k(p-1)} k^{n-1}  \int\limits_{D_j  \times \Delta ^{n-1}} dd^c  \max (  u_j  , - \frac{1}{2^{k}})  \wedge (dd^c  \max (  v  , - \frac{1}{k}) )^{n-1} 
\\& = 2^{-k(p-1)} k^{n-1}  \int\limits_{D_j  } dd^c  \max (  u_j  , - \frac{1}{2^{k}})    \int\limits_{ \Delta ^{n-1}}  (dd^c  \max (  v  , - \frac{1}{k}) )^{n-1} 
\\& = 2^{-k(p-1)} k^{n-1} . 
\end{align*}
Proposition 2.7 in \cite{TVH} implies that 
\begin{align*}
 \int\limits_{D \times \Delta ^{n-1}} (- \max(u,v))^p d \mu 
& \leq  \liminf_{j\to+\infty} \int\limits_{D_j  \times \Delta ^{n-1}} (- \max(u_j,v) )^p \sum_{k\geq 1}(dd^c \max(2^k u_j , k v,-1) )^{n}  
\\
& \leq \sum_{k\geq 1} 2^{-k(p-1)} k^{n-1} <+\infty. 
\end{align*}
This proves the claim. 
Now, assume that $\mu = NP(dd^c w)^n$ for some $w \in\mathcal  F \text{-} PSH^-(\Omega)$. 
We claim that $w\leq w_k$ in $\Omega$ for any $k\geq 1$. Indeed, let $h\geq 1$ be an integer number and define 
$$
\varphi_h(t,z) :=  \max(h\rho(t), 2^k u(t),   kv(z),-1)   , \ \ (t,z) \in D \times \Delta^{n-1}.
$$
By Proposition 2.7 in \cite{TVH} and  Corollary 2.1 in \cite{ACH}    we have 
\begin{align*}
\int_\Omega (dd^c \varphi_h)^n 
& \leq \liminf_{j\to +\infty} \int _{D_j \times \Delta^{n-1}} (dd^c \max(h\rho_j  , 2^k u_j ,  kv ,-1)  )^{n}  
\\ & =    \liminf_{j\to +\infty} \int _{D_j  }  dd^c \max(h\rho, 2^k u  ,-1)  \int _{ \Delta^{n-1}} (dd^c \max( kv ,-1) )^{n-1}  
\\ & \leq    \liminf_{j\to +\infty} \int _{D_j  }  dd^c \max( 2^k u  ,-1)  \int _{ \Delta^{n-1}} (dd^c \max( kv ,-1) )^{n-1}  
\\& = 2^k k^{n-1}. 
\end{align*}
This implies that $\varphi_h \in \mathcal E_0 (\Omega)$ and
\begin{align*}
\sup_{h\geq 1} \int_\Omega (dd^c \varphi_h)^n  
\leq  2^k k^{n-1}. 
\end{align*}
Since $w_k$ is bounded and $\varphi_h \searrow w_k$ on $\Omega$ as $h \nearrow+\infty$, we infer that   $w_k \in \mathcal F_1 (\Omega)$. 
Proposition \ref{lem0} implies that   
$w\leq w_k \text{ in } \Omega$
because $(dd^c w)^n \leq (dd^c w_k)^n $ on $\{w>-\infty\}$. 
This proves the claim.  
Letting $k\to\infty$ we arrive that $w \leq -1 $ on $\Omega$, and hence,  $w +1\in \mathcal  F \text{-} PSH^- (\Omega)$. Replace $w$ by $w+1$ and using above argument we obtain that 
$w +1 \leq -1 \text{ on } \Omega.$
Therefore,  $w \leq -2 $ on $\Omega$. By induction we obtain that  
$$
w\equiv -\infty \text{ in } \Omega .
$$
This is impossible.  Thus, there is no function $w\in\mathcal  F \text{-} PSH^-(  \Omega )$ satisfying $NP(dd^c w)^n=\mu$ on $QB( \Omega)$. The proof is complete. 
\end{proof}



%% References


%\bibliographystyle{elsarticle-num}



\end{document}